\documentclass[a4paper]{amsart}
\usepackage{amssymb}
\usepackage{graphicx}
\usepackage{caption}
\usepackage{subcaption}
\usepackage{float}
\usepackage{amsthm}
\usepackage{amsmath}
\usepackage{amsfonts}
\usepackage{amssymb}
\usepackage{graphicx}
\usepackage{tikz}
\usepackage{enumerate}

\usepackage[pdftitle={On trees with the same restricted U-polynomial and the
Prouhet-Tarry-Escott problem}]{hyperref}
\theoremstyle{plain}
\newtheorem{theorem}{Theorem}[section]
\newtheorem{lemma}[theorem]{Lemma}
\newtheorem{proposition}[theorem]{Proposition}

\newtheorem{corollary}[theorem]{Corollary}

\theoremstyle{definition}

\newtheorem{question}[theorem]{Question}

\theoremstyle{remark}

\begin{document}
\title{On trees with the same restricted $U$-polynomial and the Prouhet-Tarry-Escott problem}
\author{Jos\'e Aliste-Prieto}

\address{Departamento de Matem\'aticas, Universidad Andres Bello, Republica 220, Santiago, Chile}
\email{jose.aliste@unab.cl}
\author{Anna de Mier}
\address{Departament de Matem\`atiques, Universitat Polit\`ecnica de Catalunya
, Jordi Girona 1-3, 08034 Barcelona, Spain
}
\email{anna.de.mier@upc.edu}
\author{Jos\'e Zamora}
\address{Departamento de Matem\'aticas, Universidad Andres Bello, Republica 220, Santiago, Chile}
\email{josezamora@unab.cl}

\date{\today}
\maketitle

\begin{abstract}
This paper focuses on the well-known problem due to Stanley of whether two non-isomorphic trees
can have the same $U$-polynomial (or, equivalently, the same chromatic symmetric function). We consider the $U_k$-polynomial, which is a restricted version of $U$-polynomial, and construct with the help of solutions of the Prouhet-Tarry-Escott problem, non-isomorphic trees with the same $U_k$-polynomial for any given $k$. By doing so, we also find a new class of trees that are distinguished by the $U$-polynomial up to isomorphism. 
\end{abstract}
\section{Introduction}
The aim of this paper is to contribute towards a solution of Stanley's question of whether there exist two non-isomorphic trees with the same chromatic symmetric function. 

The $U$-polynomial (introduced by Noble and Welsh \cite{noble99weighted}) and the chromatic symmetric function $X_G$ (introduced by Stanley \cite{Stanley95symmetric}) of a graph $G$ are powerful graph isomorphism invariants. They encode much of the combinatorics of the given graph. In particular, many other well-known invariants such as the Tutte polynomial and the chromatic polynomial can be obtained as evaluations of them.

The main problem about any graph invariant is to understand which classes of graph it distinguishes. One way of approaching this problem
is by finding graphs with the same invariant. For the chromatic symmetric function, such examples already appear in \cite{Stanley95symmetric}. For the $U$-polynomial, by results of Sarmiento \cite{sarmiento2000polychromate},
one checks that work of Brylawski \cite{brylawski1981intersection} in the context of the polychromate provides  such examples. However, the following question by Stanley \cite{Stanley95symmetric} remains unsolved. 
\begin{question}[Stanley's question]
Do there exist non-isomorphic trees with the same chromatic symmetric function? 
\end{question}
We note that it is well known that Stanley's question is equivalent to the similar question for the $U$-polynomial, since, when restricted to trees, the chromatic symmetric function and the $U$-polynomial determine each other (see \cite[Theorem 6.1]{noble99weighted}). This means, in particular, that any statement related to the chromatic symmetric function of a tree can be rewritten in terms of the $U$-polynomial. In this article, we prefer to write everything in terms of the $U$-polynomial. 

There are several special classes of trees where the restriction of Stanley's question has as solution. In \cite{martin2008distinguishing}, Martin, Morin and Wagner showed that the $U$-polynomial distinguishes spider trees and a subclass of caterpillars (they also showed how to compute much of the combinatorics of a tree from the coefficients of its $U$-polynomial). In \cite{aliste2014proper}, the first and third author showed that the $U$-polynomial distinguishes the larger class of all proper caterpillars. It is still unknown to the authors whether the $U$-polynomial distinguishes non-proper caterpillars. In a different direction, in \cite{orellana2014graphs}, Orellana and Scott showed how to reconstruct a tree from a \emph{labeled} version of the $U_2$-polynomial, provided the tree has a unique vertex as a centroid (the centroid and the $U_2$-polynomial are defined later on in this article, for the meaning of what labeled means in this context, we refer the reader directly to \cite{orellana2014graphs}). More recently, in \cite{smith2015symmetric}, Smith, Smith and Tian have extended Orellana and Scott's results to show that a \emph{labeled} version of the $U_3$-polynomial suffices to reconstruct any tree. Finally, Loebl and Sereni \cite{loebl2014potts} have developed some techniques for constructing families of weighted graphs that are distinguished by the $W$-polynomial, which is the weighted version of the $U$-polynomial.

In this paper, we consider a restricted version of the $U$-polynomial, which we 
call the $U_k$-polynomial for any fixed integer $k$. Our main result is to exhibit examples of non-isomorphic trees with the same $U_k$-polynomial 
for every $k$. One of the motivations for this work comes from Orellana and Scott's results, in the sense that our examples could shed some light about possible obstructions for extending Orellana and Scott's results into a solution of the Stanley's question. This is emphasized by the fact that our examples generalize 
some of the examples already found in \cite{orellana2014graphs} (see Figure 2). Let us note that Smith, Smith and Tian \cite{smith2015symmetric} have found, with the aid of the computer, the smallest pairs of non-isomorphic trees with the same $U_k$-polynomial for $k\in\{1,2,3,4\}$.

In order to construct our examples, we reduce the problem to finding solutions of an old problem in number theory known 
as the \emph{Prouhet-Tarry-Escott problem} (PTE problem for short).  Given $k$ be a positive integer,
we ask whether there exist integer sequences  $a = (a_1,\ldots, a_n)$ and $b = (b_1,\ldots, b_n)$, distinct up to permutation, such that 
\begin{equation}
\label{eq:PTE}
\sum_{i=1}^n a_i^{\ell}=\sum_{i=1}^n b_i^{\ell} \quad \mbox{ for all } 1\leq \ell \leq k.
\end{equation}
If $a$ and $b$  are solutions for this problem  for some $k\geq 1$,  we denote it by $a =_k b$ for short. We call $k$ the \emph{degree} of the solution
and the length the sequences its \emph{size} (note that some terms could be zero). The history of the PTE problem probably goes back to Euler and Goldbach (1750-51) who noted that 
\[(a,b,c,a+b+c) =_2 (0,a+b,a+c,b+c).\]
Independently, Prouhet (1851) and Tarry and Escott (1910) showed that that for every $k$ there are solutions to the PTE-problem. For more history and results related to the PTE problem, we refer
to the reader \cite{wright1959prouhet,borwein2002computational}. We also note that once one solution for the PTE problem has been found, many other \emph{equivalent} solutions can be easily constructed. This follows from the fact that, if $a =_k b$ and $f(t) = \alpha t + \beta$ is an affine transformation with integer coefficients, then it is easy to check that $f(a) =_k f(b)$. For convenience, we usually write $\alpha a + \beta$ instead of $f(a)$. For instance, if $a =_k b$, then $a + 1 =_k b +1$ and $\alpha - a  =_k \alpha - b$ for every integer $\alpha$.

This paper is organized as follows. Background and statement of the main results is done in Section~\ref{sec:back} while Section 3 and 4 are devoted to the proofs. 

\section{Background and Results}\label{sec:back}
In this section we recall the definitions of the $U$-polynomial and then state our results. It will be convenient to first recall the definition of the $W$-polynomial for weighted graphs.  Given a graph $G$, a \emph{weight function} is a map $\omega$ from the vertices of $G$ to the positive integers. A \emph{weighted graph} is a graph $G$ endowed
with a weight function $\omega$, denoted by $(G,\omega)$. 

The $W$-polynomial of a weighted graph  was introduced in \cite{noble99weighted} by means of a \emph{deletion-contraction} formula. Four our purposes, it is easier to use its  states model representation (see \cite[Theorem 4.3]{noble99weighted}) as a definition. We need some notation first. Let $(G,\omega)$ be a weighted graph with $G=(V,E)$. The number of connected components of $G$ is denoted by $k(G)$.  Given $A\subseteq E$, the restriction $G|_A$ of $G$ to $A$ is obtained by deleting every edge that is not contained in $A$ (but keeping all the vertices). The \emph{rank} of $A$, denoted by $r(A)$, is 
defined as 
\[r_G(A) = |V| - k(G|_A).\]
The \emph{partition} induced by $A$, denoted by $\lambda_G(A)$, is the partition of $|V|$ determined by the total weight of the vertices in each connected component  of $G|_A$.  When $G$ is clear from the context, we write $\lambda(A)$ and $r(A)$ instead of $\lambda_G(A)$ and $r_G(A)$. Let $\mathbf{x} = x_1,x_2,\ldots$ be an infinite set of commuting indeterminates. Given any partition $\lambda$, we encode it as the monomial  $\mathbf{x}_\lambda:=x_{\lambda_1}\cdots x_{\lambda_l}$. The {\em $W$-polynomial} of  $(G,\omega)$  is defined as 
\begin{equation}
\label{def:u_poly}
W(G,\omega;\mathbf x, y)=\sum_{A\subseteq E}\mathbf x_{\lambda(A)}(y-1)^{|A|-r(A)}.
\end{equation}
If $G$ is a (unweighted) graph, then the \emph{$U$-polynomial} of $G$ is defined
as the $W$-polynomial of $(G,1_V)$ where $1_V$ is the weight function that gives weight $1$ to each vertex of $G$. Here, $\lambda(A)$ reduces to the partition induced
by the number of vertices on each connected component of $G|_A$.

In what follows, all graphs are assumed to be trees. In this case, it is easy to check that $r(A)=|A|$ for every $A\subseteq E$. It follows that the $U$-polynomial of a tree $T$ can be rewritten as
\[U(T) = U(T;\mathbf x) = \sum_{A\subseteq E}\mathbf x_{\lambda(A)} = 
\sum_{A\subseteq E}\mathbf x_{\lambda(E\setminus A)}.\]
In this note, we focus on the following restricted-version of the $U$-polynomial. Given a positive integer $k$, let 
\[U_{k}(T) = U_{k}(T;\mathbf x) = \sum_{A\subseteq E,|A| \leq k}\mathbf x_{\lambda(E\setminus A)},\]
that is, we restrict the cardinality of the edge sets appearing in the definition of the $U$-polynomial. Of course, if $k = |E|$, then $U(T;\mathbf{x})=U_k(T;\mathbf{x})$. 

The main goal of this note is to exhibit examples of non-isomorphic trees with the same $U_k$-polynomial for any given $k$. To do so, we first introduce a new class of trees encoded by non-negative integer sequences. As a convention, the term \emph{$n$-star} will refer to a star $K_{1,n-1}$ with $n$ vertices, and the term \emph{$n$-path} will refer to a path with $n$ vertices. 
\begin{figure}[t]
\begin{minipage}{5in}
  \centering
  \raisebox{-0.5\height}{\includegraphics[scale=0.4]{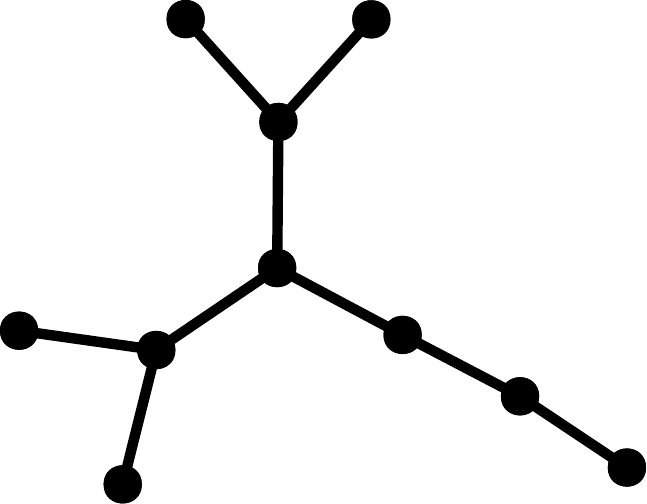}}
  \hspace*{.4in}
  \raisebox{-0.5\height}{\includegraphics[scale=0.4]{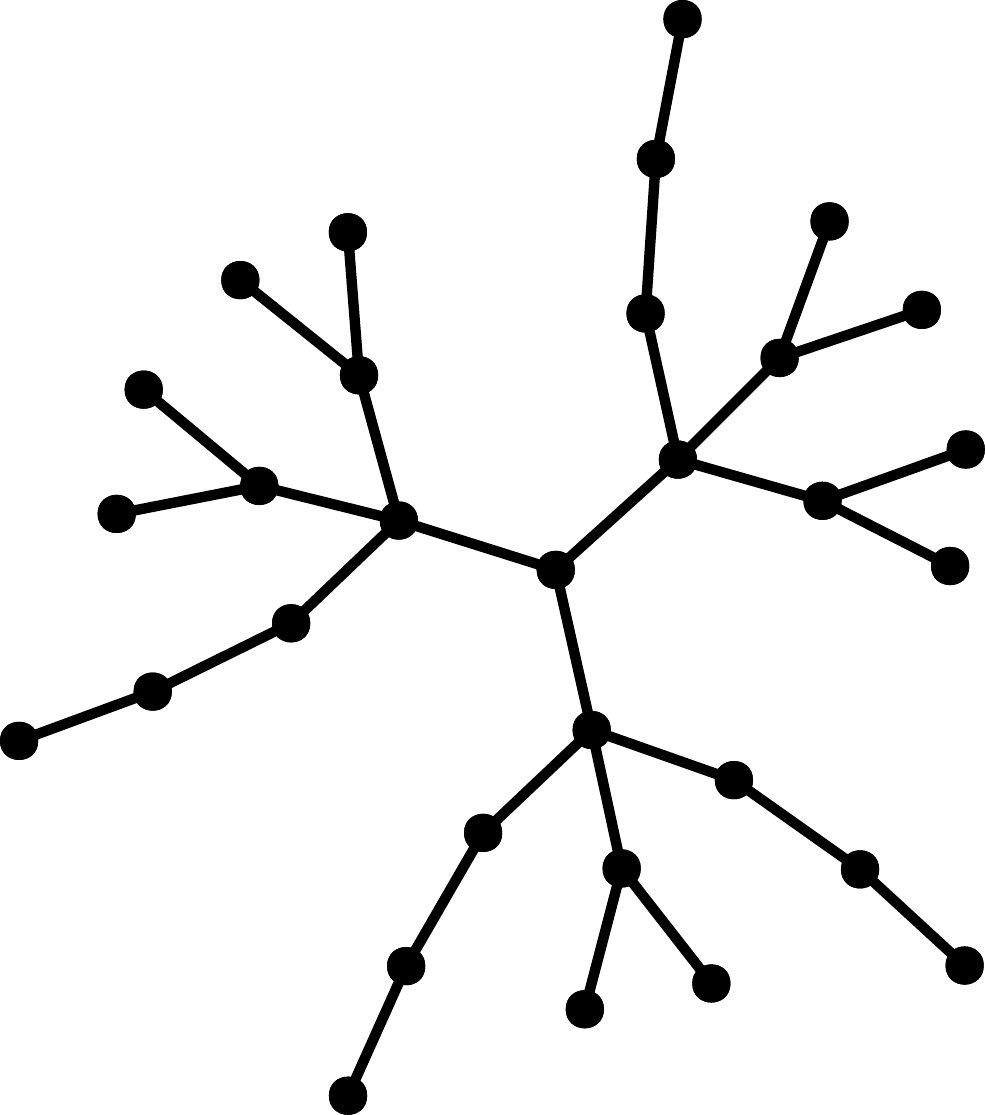}}
\end{minipage}
\caption{The graph on the left is $B(1,2)$, the graph on the right is $T_3(1\ 1\ 2) = T(1\ 1\ 2, 2\ 2\ 1 )$}
\label{fig:B}
\end{figure}

Given two non-negative integers $p$ and $s$, the tree $B(p,s)$ is the tree formed by taking  $p$ disjoint copies of a $4$-path and $s$ disjoint copies of a $4$-star, and then 
identifying one leaf-vertex of each copy into a common vertex $v$. The vertex $v$ is considered as the root of $B(p,s)$. Note that $B(0,0)$ consists of an isolated vertex $v$ (See Figure \ref{fig:B} for an example). 

Next, given two sequences $p=(p_1,\ldots,p_n)$  and $s=(s_1,\ldots,s_n)$ of non-negative integers with length $n\geq 2$, the tree $T(p,s)$ is constructed as follows.  Take  the disjoint union of $B(p_i,s_i)$ for all $1\leq i\leq n$ and denote their respective roots by $v_i$. Next, join each vertex $v_i$ to a central vertex $c$. This yields $T(p,s)$. The  subgraph induced on the vertices $\{c,v_1,\ldots,v_n\}$ will be referred to as the \emph{core} of $T(p,s)$. It is easy to see that the core is isomorphic to a $(n+1)$-star (See also Figure \ref{fig:B} for an example). 

\begin{figure}[t]
\centerline{\includegraphics[scale=0.4]{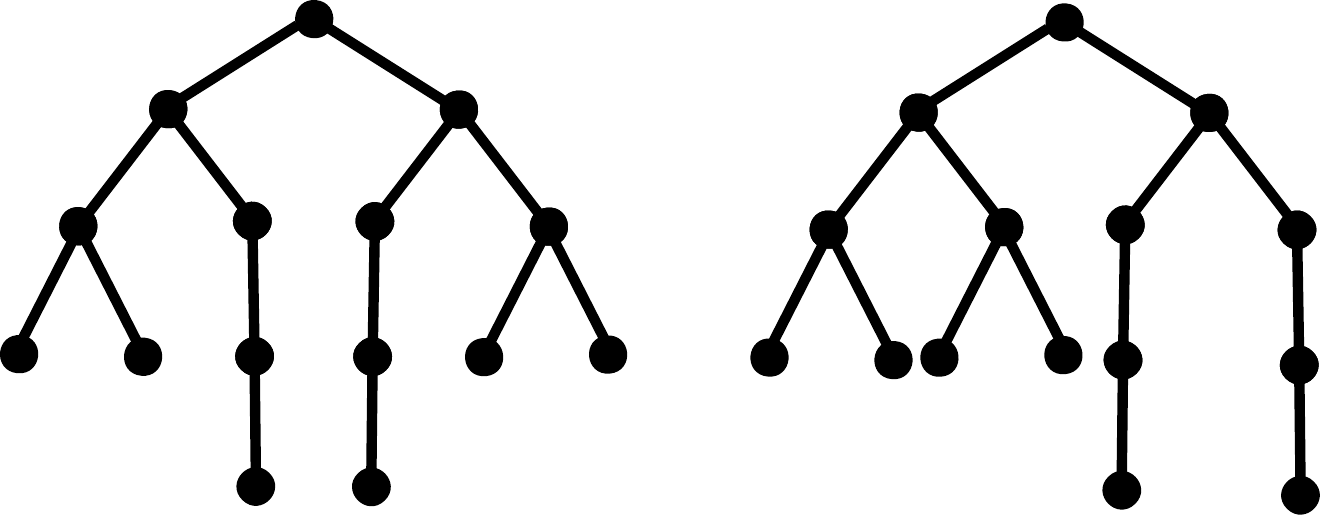}}
\caption{The graphs $T_2(1\ 1)$ and $T_2(2\ 0)$ have the same $U_2$-polynomial but distinct $U_3$-polynomial. Note
that $1+1 = 2+0$ but $1^2+1^2\neq 2^2+0^2$.}
\label{fig:example}
\end{figure}

Finally, we consider a special case of the last construction. Given a positive integer $\alpha$, we say that a sequence of non-negative integers $p =(p_1,\ldots, p_n)$ is $\alpha$-compatible if $n\geq 2$ and $\max_i p_i \leq \alpha$. If $p$ is an $\alpha$-compatible
sequence, then $\alpha - p$ denotes the sequence $(\alpha-p_1,\ldots, \alpha-p_n)$. Define $T_\alpha(p) := T(p,\alpha-p)$. We say that a tree $T$ is a \emph{Prouhet-Tarry-Escott tree} (for short PTE-tree) if there exist $\alpha$ and $p$ such that $T$ is isomorphic to $T_\alpha(p)$.

Our first result, to be proved in  Section~\ref{sec:sameUk},  is the following:
\begin{theorem}
\label{thm:PTEencodes}
Two non-isomorphic $PTE$-trees have the same $U_{k+1}$-polynomial if and only if their associated 
sequences are solutions for the PTE-problem with degree $k$. 
\end{theorem}
A simple example of application of Theorem \ref{thm:PTEencodes} can be seen in Figure \ref{fig:example}. To construct our examples, by Theorem \ref{thm:PTEencodes}, we
need to find solutions for the PTE-problem for any degree. As noted
in the introduction, the existence of such solutions  was already known to Prouhet. In \cite{wright1948equal}, Wright proved the following stronger result:
\begin{theorem}
\label{thm:Prouhet}
For $k\geq 1$, $j\geq 2$ there exist sequences $a_1,a_2,a_3,\ldots,a_j$ of length $n\leq (k^2+k+2)/2$, distinct up to permutation,
such that
\[a_1 = _k a_2 =_k  \cdots =_k a_j.\]
\end{theorem}
The latter theorem allows us to obtain sets of any given cardinality of non-isomorphic trees with same
$U_k$-polynomial.
\begin{corollary}
For every pair of positive integers $k$ and  $j$, there is a set of $j$ non-isomorphic trees with the same $U_k$-polynomial. 
\end{corollary}
It is well-known that if two sequences of length $n$ satisfy $p=_kp'$, then $n>k$ (see \cite[Prop. 2]{borwein1994prouhet}). Thus, 
as another consequence of Theorem \ref{thm:PTEencodes} we get 
\begin{corollary}
Let $T$ and $T'$ be two PTE-trees. Then $T$ and $T'$ have the same $U$-polynomial if and only if they are isomorphic. 
\end{corollary}
By the last corollary, it is natural to ask whether the $U$-polynomial distinguishes PTE-trees from non-PTE trees. This is indeed the case.
\begin{proposition}
\label{prop:recog}
The $U_1$-polynomial recognizes whether a tree is a PTE-tree or not. 
\end{proposition}
The proof of this proposition uses some techniques from \cite{orellana2014graphs} and it is given in Section \ref{sec:recognizing}. As a direct corollary of this proposition
and Theorem \ref{thm:PTEencodes} we get that all PTE-trees are distinguished up to isomorphism by the $U$-polynomial:
\begin{corollary}
If $T$ is a PTE-tree, and $T'$ is another tree such that $U(T) = U(T')$, then $T'$ is isomorphic to $T$. 
\end{corollary}

\section{Solutions to the PTE problem of degree $k$  induce trees with same $U_{k+1}$-polynomial}\label{sec:sameUk}

In this section, we will prove Theorem \ref{thm:PTEencodes}.
Let $T$ be a tree and $q$ and $t$ two non-negative sequences of the same length. If $S$ is a subtree of $T$ isomorphic to $T(q,t)$, we say that $S$ is of type $(q,t)$. We denote by $S_{q,t}(T)$ the set of all subtrees of $T$ of type $(q,t)$. If $T$ is a PTE-tree associated with a sequence of length $n$, then sequences $q,t$ will also assumed to be of length $n$. 

\begin{lemma}\label{lemma:sym_poly}
Let $\alpha$ be a non-negative integer and  $q,t$ be two non-negative integer sequences of length $n$. Then, there is a symmetric polynomial $P_{\alpha,q,t}(x_1,x_2,\ldots,x_n)$ of
degree at most $\sum_i(q_i+t_i)$ such that, for every $\alpha$-compatible sequence $p$ of length $n$, 
we have 
\[ | S_{q,t}(T_\alpha(p))| = P_{\alpha,q,t}(p_1,p_2,\ldots,p_n).\]
\end{lemma}
\begin{proof}

Let $\pi$ be a permutation of $[n]$. Let $S^{\pi}_{q,t}(T_{\alpha}(p))$ be the collection of  subtrees of $T_{\alpha}(p)$ of type $(q,t)$ such that, for all $1\leq i\leq n$,  vertex $v_i$ is the root of a subtree isomorphic to $B(q_{\pi(i)},t_{\pi(i)})$. Clearly 
\[S_{q,t}(T_\alpha(p))=\bigcup_{\pi\in S_n} S^{\pi}_{q,t}(T_{\alpha}(p)), \quad 
| S^{\pi}_{q,t}(T_{\alpha}(p))|=\prod_{i=1}^n\binom{p_i}{q_{\pi(i)}}\binom{\alpha-p_i}{t_{\pi(i)}}.\]
Notice also that $S^{\pi}_{q,t}(T_{\alpha}(p))=S^{\pi'}_{q,t}(T_{\alpha}(p))$ only when $q_{\pi(i)}=q_{\pi'(i)}$ and $t_{\pi(i)}=t_{\pi'(i)}$ for all $1\leq i\leq n$. With $\pi$ fixed, the number of such permutations $\pi'$ depends only on the symmetries of $(q,t)$. More concretely, it equals the number of permutations $\sigma\in S_n$ such that $q_{\sigma(i)}=q_{i}$ and $t_{\sigma(i)}=t_{i}$ for all $1\leq i\leq n$. Let us denote this number by $N_{q,t}$.

Thus, let  
\begin{equation}\label{eq:subtrees}
P_{\alpha,q,t}(x_1,x_2,\ldots) = \frac{1}{N_{q,t}} \sum_{\pi \in S_n}  \prod_{i=1}^n {x_{i} \choose  q_{\pi(i)} }  {\alpha - x_{i} \choose  t_{\pi(i)}}.
\end{equation}
It is clear that $P_{\alpha,q,t}$ is symmetric of degree at most $\sum_i (q_i+t_i)$, and by the discussion above $P_{\alpha,q,t}(p_1,p_2,\ldots)$ is the required number of subtrees.
\end{proof}

\begin{corollary}\label{cor:samesubtrees}
Let $\alpha$ be a positive integer. Suppose that  $p$ and $p'$  are two $\alpha$-compatible sequences such that $p =_k p'$. Then,  for every  $({q},{t})$ such that  $\sum_i ( q_i +t_i) \leq k$  we have 
\[ | S_{q,t}(T_\alpha(p))| = | S_{q,t}(T_\alpha(p'))|.\]
\end{corollary}

\begin{proof}
By Lemma \ref{lemma:sym_poly}, there exists a symmetric polynomial $P(x_1,x_2,\ldots,x_n)$ of degree less or equal than $k$ such that 
$P(p_1,p_2,\ldots,p_n) = |S_{q,t}(T_\alpha(p))|$ and $P(p'_1,p'_2,\ldots,p'_n) = | S_{q,t}(T_\alpha(p'))|$. By \cite[Corollary 7.7.2]{stanley1999enumerative}, this polynomial can be written as a linear combination of the power sum symmetric polynomials of degree less or equal than $k$. Since $p =_k p'$, the conclusion follows.
\end{proof}

Let $S$ be a subtree of $T$.
Then, the contraction $S_{\omega,T}$ of $S$ in $T$ is the weighted
tree obtained by contracting all the edges not in $S$ and adding weights along contracted edges. If $F$ is a subset of edges of $T$, define 
\[U_F(T) = \sum_{A\subseteq F} \mathbf{x}_{\lambda(E\setminus A)}.\]
The proof of the lemma below follows directly from the definitions.
\begin{lemma}
\label{lemma:contracting}
Let $T=(V,E)$ be a tree and $S=(W,F)$ be a subtree of $T$. For every $A\subseteq F$, we have 
\begin{equation}\label{eq:weight}
\lambda_T(E\setminus A) = \lambda_{S_{\omega,T}}(F\setminus A).
\end{equation}
Moreover, $U_F(T) = W (S_{\omega,T})$.
\end{lemma}

\begin{proposition}
\label{cor:core}
Let $T = T_\alpha(p)$ and $T' = T_\alpha(p')$ for two sequences $p$ and $p'$ such that $p=_1p'$. If $K$ and $K'$ denote, respectively, the core of $T$ and $T'$, then
\[U_K(T) = U_{K'}(T')\quad\text{and}\quad U_{E\setminus {K}}(T) = U_{E\setminus K'}(T').\]
\end{proposition}
\begin{proof}
Both assertions follow from Lemma \ref{lemma:contracting} after observing that   $K_{\omega,T}$ and $K'_{\omega,T'}$ are isomorphic weighted graphs and, also  $(E\setminus K)_{\omega,T}$ and $(E'\setminus K')_{\omega,T'}$ are isomorphic weighted graphs because  $p  =_1 p' $ and $\alpha - p  =_1 \alpha - p'$. 
\end{proof}

\begin{theorem}\label{thm:sameU}
Let $T = T_\alpha(p)$ and $T' = T_\alpha(p')$ for two $\alpha$-compatible sequences $p$ and $p'$ such that $p =_k p'$. Then 
\[U_{k+1}(T) = U_{k+1}(T').\]
\end{theorem}
\begin{proof}
We give a bijection $\varphi$ between edge-subsets of $T$ and $T'$ of size at most $k+1$ such that $A$ and $\varphi(A)$ induce the same partition in $T$ and $T'$, respectively. By Corollary \ref{cor:core}, it suffices to define $\varphi$ on subsets of edges that intersect both the core of $T$ and its complement. 

Let $$\mathcal{S}_T=\bigcup_{\stackrel{(q,t)}{\sum_i(q_i+t_i)\leq k}} S_{q,t}(T),$$ and define $\mathcal{S}_{T'}$ analogously.
By Corollary~\ref{cor:samesubtrees}, there is a type-preserving bijection $\Phi:\mathcal{S}_T\rightarrow \mathcal{S}_{T'}$. Also, for each $S\in \mathcal{S}_{T}$ we fix an isomorphism between $S$  and $\Phi(S)$; this isomorphism is also a weighted-graph isomorphism between $S_{w,T}$ and $(\Phi(T))_{w,T'}$. 

Let $A$ be a subset of edges of $T$ with $|A|\leq k+1$ that intersects the core of $T$ and its complement. Since $A$ has at most $k$ edges in the complement of the core, it is contained in some subtree belonging to  $\mathcal{S}_T$. Let $S_A$ be the smallest such subtree, which is well defined since the intersection of elements of $\mathcal{S}_T$ is again in $\mathcal{S}_T$. Let $\varphi(A)$ be the edge-subset of $T'$ to which $A$ is mapped under the fixed bijection between $S_A$ and $\Phi(S_A)$. By construction, $A$ and $\varphi(A)$ contribute  the same term to $U_{k+1}(T)$ and $U_{k+1}(T')$, respectively.

To finish the proof it is  enough to notice that $\Phi(S_A)$ is the smallest subtree in $\mathcal{S}_{T'}$ that contains $\varphi(A)$, and thus $A$ can be recovered from $\varphi(A)$ and $\varphi$ is a bijection.

\end{proof}
\begin{proposition}\label{prop:notsameU}
Let $T = T_\alpha(p)$ and $T' = T_\alpha(p')$ for two sequences $p$ and $p'$ such that 
\[  p=_k p' \ \mbox{ but }\ \sum_i p_i^{k+1}\neq \sum_i (p_i')^{k+1}.\]
Then 
\[U_{k+2}(T) \neq U_{k+2}(T').\]
\end{proposition}

\begin{proof}
We start by showing that we must have $k+1\leq \alpha$. It is easy to see that $p=_kp'$ implies that $(x-1)^{k+1}$ divides $\sum_i (x^{p_i}-x^{p_i'})$ (see for instance\cite[Prop.1]{borwein1994prouhet}). This forces $k+1\leq \max \{\max_i p_i, \max_i p_i'\}\leq \alpha$.

Let $N=n(3\alpha+1)+1$ be the number of vertices of $T$ and $T'$. We claim that the coefficients of $x_2^{k+1}x_{3\alpha-1-2k} x_{N-3\alpha-1}$ in $U(T)$ and $U(T')$ are different (note that since $k+1\leq \alpha$ and $n\geq 2$ by assumption, the expressions $x_{3\alpha-1-2k}$ and  $x_{N-3\alpha-1}$ are indeed variables). The only way to obtain such a coefficient in $U(T)$ is by removing one of the edges $cv_i$ and $k+1$ among the $p_i$ edges adjacent to $v_i$, and the same applies to $T'$.
Hence,
\begin{align*}
& [x_2^{k+1}x_{3\alpha-1-2k} x_{N-3\alpha-1}]U(T)=\sum_{i=1}^n \binom{p_i}{k+1}, \\
& [x_2^{k+1}x_{3\alpha-1-2k} x_{N-3\alpha-1}]U(T')=\sum_{i=1}^n \binom{p'_i}{k+1}.
\end{align*}
The difference $\sum_{i=1}^n \binom{p_i}{k+1}-\sum_{i=1}^n \binom{p'_i}{k+1}$ can be expressed as a linear combination of the power-sum symmetric polynomials, with the highest degree term being $\sum_i(p_i^{k+1}-(p'_i)^{k+1})/(k+1)!$. Thus, it is clear that both coefficients are different.
\end{proof}

Theorem~\ref{thm:sameU} and Proposition~\ref{prop:notsameU} together establish Theorem~\ref{thm:PTEencodes}.

\section{Recognizing PTE-trees with the $U$-polynomial}
\label{sec:recognizing}
In this section we show that the $U_1$-polynomial recognizes whether a tree is PTE or not. To see this, we recall some techniques introduced by Orellana and Scott \cite{orellana2014graphs}.
Let $T$ be a tree. Given any vertex $v$ in $T$, a \emph{branch} of $v$ is any subtree of $T$ having $v$ as a leaf-vertex. Then, the \emph{branch-weight} of $v$ is the maximum number of edges in any branch of $v$. The \emph{centroid} is defined as the set of vertices with minimum branch-weight. It is known that the centroid contains either one or two vertices (in which case they are connected by an edge). 
We say that two edges  $e$ and $f$ \emph{attract} if the unique path joining $e$ and the centroid passes through $f$ or the unique path joining $f$ with the centroid passes through $e$. Otherwise, we say that $e$ and $f$ \emph{repel}.
When $T$ has a unique centroid $c$, we set $c$ as the root of $T$ and label the edges of $T$ as follows. For each edge $e$ in $T$, its label $\theta_e$ is the unique positive integer satisfying $\lambda_T(E\setminus \{e\})=(N-\theta_e,\theta_e)$ with $N-\theta_e\geq \theta_e$, where $N$ denotes the number of vertices of $T$. We denote by  $\mathcal{M}_T$ the  multiset of labels of $T$. The following lemma summarizes the tools from \cite{orellana2014graphs} that we need here.
\begin{lemma}
\label{lemma:label_compatible}
Let $T$ be a tree labeled as above.  Then the labels along any path starting at the centroid are strictly decreasing. Moreover, for each edge $e$, its label is the the sum of the labels of its child-edges plus one. 
In particular, edges with the same labels always repel. 
\end{lemma}

The following result is a more precise statement for  Proposition \ref{prop:recog}.
\begin{proposition}
Let $\alpha$ be a positive integer, $p$ be an $\alpha$-compatible sequence of length $n$ and $T = T_\alpha(p)$ the associated PTE-tree. Let $N = (3\alpha +1)n+1$ be the total number of vertices of $T$ and  $\beta = \sum_i p_i$. Then, 
the following assertions hold:
\begin{enumerate}[i]
\item We have $\mathcal{M}_T = \{1^{2n\alpha-\beta},2^{\beta}, 3^{n\alpha}, (3\alpha+1)^{n}\}$.
\item If $T'$ is another tree with $U_1(T)=U_1(T')$, then  $T'$ is a PTE-tree $T_\alpha(p')$ with $\sum_i p'_i = \beta$.
\end{enumerate}
\end{proposition}
\begin{proof}
The first assertion follows directly from the construction of $T_\alpha(p)$. 

For the second assertion, first recall that a tree  has a unique centroid if and only if there is no edge with label $N/2$, where $N$ is the number of vertices. That is, the property of having a unique centroid is determined by the $U_1$-polynomial of the tree. In particular, since $T$ has a unique centroid, so has $T'$. Since edges with the same label repel, it is clear that the $n$ edges with label $3\alpha+1$ must be incident to the centroid. Since $N = (3\alpha+1)n+1$, it follows from Lemma \ref{lemma:label_compatible}.that no other edge may be incident to the centroid. Thus, each edge with label $3$ must be attached to one edge of label $3\alpha+1$. Since we have $n\alpha$ such edges and $n$ edges of label $3\alpha+1$, again by Lemma \ref{lemma:label_compatible}, each 
edge with label $3\alpha+1$ has exactly $\alpha$ child-edges with label $3$ and cannot have any other child-edges. By now, edges with label $2$ and $1$ can only be
incident with edges of label $3$. 

It is easy to see that each edge with label $3$ is either incident to two edges of label $1$ or to one edge of label $2$, which in turn is incident to an edge of label $1$.
For each $e_i$ with label $3\alpha+1$ set $p'_i$ be the number of edges of label $3$ of the second kind that are incident to $e_i$. It is clear 
that $\sum_i p_i'=\beta$ and that $T' = T_\alpha(p')$ and the proof is finished.  
\end{proof}

\section*{Acknowledgments}
The first author is partially supported by FONDECYT 11121510. The second author is partially supported by projects MTM2014-54745-P and Gen. Cat. DGR 2014SGR46. She also thanks Josep M. Brunat for fruitful discussions about the Prouhet-Tarry-Escott problem. The third author is partially supported by N\'ucleo Milenio Informaci\'on y Coordinación en Redes ICM/FIC RC130003. The first and third author are partially supported by Basal PFB-03 CMM, Universidad de Chile.

\end{document}